\documentclass[11pt,reqno]{article}
\usepackage{amsmath,amsthm,amssymb,cite}
\usepackage{tikz}
\usetikzlibrary{calc}
\usetikzlibrary{arrows,chains,matrix,positioning,scopes}
\makeatletter
\tikzset{join/.code=\tikzset{after node path={%
\ifx\tikzchainprevious\pgfutil@empty\else(\tikzchainprevious)%
edge[every join]#1(\tikzchaincurrent)\fi}}}
\makeatother
\tikzset{>=stealth',every on chain/.append style={join},         every join/.style={->}}

\newtheorem{theorem}{Theorem}[section]

\newtheorem{lemma}[theorem]{Lemma}
\newtheorem{corollary}[theorem]{Corollary}
\theoremstyle{definition}

\title{On the Burau Representation of $B_4$ modulo $p$}
\author{\footnote{Authors are ordered alphabetically } A. Beridze, S. Bigelow, P. Traczyk}

\begin{document}

\maketitle

\begin{abstract} The problem of faithfulness of the (reduced) Burau
representation for $n =4$ is known to be equivalent to the problem
of whether certain two matrices $A$ and $B$ generate a free group
of rank two. It is known that $A^3$ and $B^3$ generate a free
group of rank two \cite{9}, \cite{10}, \cite{4}.  We prove that
they also generate a free group when considered as matrices
over the $\mathbb{Z}_p[t,t^{-1}]$ for any integer $p > 1$.
\end{abstract}

\section{Introduction}

The faithfulness problem of  the Burau representation is open  in
just one case, that of $n=4$ \cite{2}, \cite{7}, \cite{8}. On the
other hand, for $n=4$ the representation is faithful if and only
if certain two matrices $A$ and $B$ generate a free group
$\langle A,B \rangle$ of rank two \cite{1}, \cite{4}.

Let $ \rho_4 :B_4\to GL\left(3,\mathbb{Z}\left[t,t^{-1}\right]\right)$
be  the reduced Burau representation of the braid group $B_4$. 
Consider the matrices $A$ and $B$ as follows:
\begin{equation} \label{1}
A=\rho_4 \left(a^{-1}\right)=\left[ \begin{array}{ccc}
0 & 0 & -t^{-1} \\ 
0 & -t & -t^{-1}+t \\ 
-1 & 0 & -t^{-1}+1 \end{array}
\right],
\end{equation}
\begin{equation} \label{2}
~~~~  B=\rho_4 \left(b\right)=\left[ \begin{array}{ccc}
-t^{-1} & 1 & 0 \\ 
0 & 1 & 0 \\ 
0 & 1 & -t \end{array}
\right],
\end{equation}
where $a={\sigma }_1{\sigma }_2{\sigma }^{-1}_1{\sigma }_3{\sigma }_2^{-1}{\sigma }_1^{-1}$ and $b={\sigma }_3{\sigma }^{-1}_1$ ($\sigma_i, ~i=1,2,3$ are standard generators of $B_4$). It is known that the group $\langle a,b\rangle$ generated by $a$ and $b$ is a free group  which contains the kernel of the Burau map $\rho_4 : B_4 \to G
L\left(3, Z[t,t^{-1}]\right)$ \cite{6}, \cite{4}. 

In the paper it is shown that there exists an order four matrix $T$, which satisfies the following equality:
\begin{equation}
A =T B T^{-1}, ~~~ A^{-1}=T^{-1}BT, ~~~ B^{-1}=T^2 BT^2.
\end{equation}
Consequently, by a suitable substitution in a word 
\begin{equation} \label{5}
\omega=A^{m_k}B^{n_k}\dots A^{n_2}B^{n_2}A^{m_1}B^{n_1},~~
n_i, m_i \in \mathbb{Z}
\end{equation}
we obtain a word in $T$, $T^{-1}$ and $B^n,~ n \in \mathbb{N}$. Therefore, the faithfulness problem of the 
Burau representation for $n=4$ reduces to showing that any word in which $B$ appears only with positive exponents, and $T$ and $T^{-1}$ appear one at a time, does not give the identity matrix. We will give a simple proof of this, which works equally well if the coefficients are in the ring $\mathbb{Z}_p[t,t^{-1}]$ for any integer $p > 1$.

\section{Matrices conjugating $B^{-1}, A$ and $A^{-1}$  to $B$}

\begin{lemma} There exists a matrix $T$ which satisfies the following equality:
	\begin{equation}
	A =T B T^{-1}, ~~~ A^{-1}=T^{-1}BT, ~~~ B^{-1}=T^2 BT^2.
	\end{equation}
	This matrix $T$ is of order four as an element of
	$GL\left(3,\mathbb{Z}\left[t,t^{-1}\right]\right)$.
\end{lemma}

\begin{proof}
	It is easily checked that the matrices $A,B, A^{-1}$ and $B^{-1}$ have the same eigenvalues
	$-t^{-1}, -t$ and $1.$
	Therefore, they  are conjugate to the same diagonal matrix:
	\begin{equation}
	\Delta=\left[ \begin{array}{ccc}
	-t^{-1} & 0 & 0 \\ 
	0 & 1 & 0 \\ 
	0 & 0 & -t \end{array}
	\right].
	\end{equation} 
	Let us consider transformation matrices 
	$T_A$ and $T_B$ 
	\begin{equation}
	T_A=\left[ \begin{array}{ccc}
	-1 & t^{-1} & 0 \\ 
	-1 & t^{-1}-1 & 1 \\ 
	-1 & -1 & 0 \end{array}
	\right], ~~~~  T_B=\left[ \begin{array}{ccc}
	1 & 1 & 0 \\ 
	0 & t^{-1}+1 & 0 \\ 
	0 & t^{-1} & 1 \end{array}
	\right].
	\end{equation}
	so that the following equalities hold.
	\begin{equation}
	A=T_A \Delta T_A^{-1}, ~~~B=T_B \Delta T_B^{-1}. 
	\end{equation}
	We obtain that
	\begin{equation} 
	A=T_A T_B^{-1} B T_B T_A^{-1}.
	\end{equation}
	If we calculate the matrix $T=T_A T_B^{-1}$, then we obtain that:
	\begin{equation}
	T=\left[ \begin{array}{ccc}
	-1 & 1 & 0 \\ 
	-1 & 0 & 1 \\ 
	-1 & 0 & 0 \end{array}
	\right], T^{-1}=\left[ \begin{array}{ccc}
	0 & 0 & -1 \\ 
	1 & 0 & -1 \\ 
	0 & 1 & -1 \end{array}
	\right].
	\end{equation}
	The direct calculation shows that the matrix $T$ is an element of order four in $GL\left(3,\mathbb{Z}\left[t,t^{-1}\right]\right)$ and   $B^{-1}=T^2BT^2$. 
\end{proof}

As we mentioned above the problem of faithfulness of the Burau representation for $n = 4$ is equivalent to the problem of
whether $A$ and $B$ generate a free group of rank $2$ in $GL\left(3,\mathbb{Z}\left[t,t^{-1}\right]\right)$ \cite{1}, \cite{4}.
This means that we need to prove that for every non--empty non--reducible word in letters $A,B,A^{-1}$ and $B^{-1}$ the corresponding product of matrices $A,B,A^{-1}$ and $B^{-1}$ is not equal to the unit matrix. If a word like this does exist we may as well consider one with suffix $B^{-i}, i \geq 1$ replacing the given word with a conjugate if necessary. However, for such words we have the following.
\begin{corollary}
	Let $w$ be a formally irreducible non--empty word in letters $A,B,A^{-1}$ and $B^{-1}$ with suffix  $B^{-i}, i \geq 1$. Then the corresponding product of matrices 
	$A,B,A^{-1}$ and $B^{-1}$ may be written in the form\\
	\begin{equation} \label{7}
	 T^{m} B^{n_k}  \cdots T^{m_2} B^{n_2} T^{m_1} B^{n_1}T^2 , ~ n_i \in \mathbb{N}, ~m_i=\pm 1, ~ m \in \{-1,0,1,2\}.
	\end{equation}
\end{corollary}
\begin{proof} We replace $A, A^{-1}$ and $B^{-1}$ with $T^{-1}BT, TBT^{-1}$ and $T^2BT^2$ 
	respectively. In the process we may obtain $TT^2, T^{-1}T^{2}, T^2T$ and $T^2T^{-1}$ between two consecutive powers of $B$ but these are equal to
	$T^{-1}, T, T^{-1}$ and $T$ respectively (because $T^4 = 1$). The need to allow $m = 2$ or $0$ arises from the possibility that the last multiplication might be by $B^{-1}$ or by $B$.
\end{proof}

\section{The group $\langle A^3,B^3 \rangle$ modulo $p$ }

Let $q_p:GL\left(3,\mathbb{Z}\left[t,t^{-1}\right]\right) \to GL\left(3,\mathbb{Z}_p\left[t,t^{-1}\right]\right)$ be the map given by reducing coefficient modulo $p$. Let the map
\begin{equation} \label{6}
\rho^p_4 :B_4\to GL\left(3,\mathbb{Z}_p\left[t,t^{-1}\right]\right)
\end{equation}
be given by $\rho^p_4=q_p \circ \rho_4$. Note that we use same notation of image matrices under the maps $\rho_4$ and $\rho^p_4$. If it is necessary, we just say that a given matrix is over the $\mathbb{Z}_p\left[t,t^{-1}\right]$. Therefore, our aim is to show that any combination of $T$, $T^{-1}$  (appearing one at a time) and positive powers of $B^n$ does not give the identity matrix over the $\mathbb{Z}_p\left[t,t^{-1}\right]$.

Note that  from now on $T$, $T^{-1}$ and  $B$ are considered as matrices over $\mathbb{Z}_p\left[t,t^{-1}\right]$, which means that coefficients of polynomials are considered modulo $p$, where $p$ is an integer greater than $1$.

\begin{theorem} Let $\omega$ be a word of the letters $B$, $T$ and $T^{-1}$ of the following form 
	\begin{equation} \label{7}
	\omega= T^{m} B^{n_k}  \cdots T^{m_2} B^{n_2} T^{m_1} B^{n_1}T^2 , ~ n_i \in \mathbb{N}, ~m_i=\pm 1, ~ m \in \{-1,0,1,2\}.
	\end{equation}
	If for every $i$ we have $n_i \geq 2,$ whenever $m_{i-1}=1$ and $n_i \geq 3,$ whenever $m_{i-1}=-1$ then the product matrix is not identity.
\end{theorem}
\begin{proof} To prove the statement, we use the method of  the ping--pong lemma. Let $X_1$ be the set of such vectors that lowest degree of the first coordinate is smaller by $2$ or more than lowest degree of second and third coordinates. Similarly, $X_2$ is the set of such vectors that lowest degree of all coordinates are equal and $X_3$ is the set of such vectors that lowest degree of second coordinate is two less than lowest degree of first and third coordinates. Note that if a coordinate is zero (trivial polynomial) that it is considered that the lowest degree is $+\infty$. Let $B$, $T$ and $T^{-1}$ acts on $X_1$, $X_2$ and $X_3$ by left multiplication. Our aim is to prove that:
\begin{enumerate}
\item $T^2 v_0 \in X_1$, for some $v_0  \notin X_1 \bigcup X_2 \bigcup X_3$;
\item $T X_1 \subset X_2$;
\item $T^{-1} X_1 \subset X_3$;
\item $B^{n} X_2 \subset X_1, ~\forall~ n\ge 2$;
\item $B^{n} X_3 \subset X_1, ~\forall~ n\ge 3$.
\end{enumerate}
In this case, for $\omega$ as in the previous theorem,
$\omega v_0 \ne v_0$ and so it is not the identity matrix.
We now prove each of these five items.

{\bf 1.}
Let $v_0= \left[ \begin{smallmatrix} 0 \\  0 \\ 1   \end{smallmatrix}\right]$, then  $T^2 v_0 =\left[ \begin{smallmatrix} 1 \\  0 \\ 0   \end{smallmatrix}\right]\in X_1$.

{\bf 2.}
Let $v \in X_1$, then it has the following form:
\begin{equation} \label{8}
v=\left[ \begin{array}{ccc}
a_0t^{n}+a_1t^{n+1}+a_2t^{n+2}+a_3t^{n+3}+\dots  \\ 
~~~~~~~~~~~~~~~~~~~~~b_2t^{n+2}+b_3t^{n+3}+\dots  \\ 
~~~~~~~~~~~~~~~~~~~~~c_2t^{n+2}+c_3t^{n+3}+\dots  \end{array}
\right]
\end{equation}
where $a_0\neq 0 \in \mathbb{Z}_p$. Then we have:
 \[
Tv=T\left[ \begin{array}{ccc}
a_0t^{n}+a_1t^{n+1}+a_2t^{n+2}+a_3t^{n+3}+\dots  \\ 
~~~~~~~~~~~~~~~~~~~~~b_2t^{n+2}+b_3t^{n+3}+\dots  \\ 
~~~~~~~~~~~~~~~~~~~~~c_2t^{n+2}+c_3t^{n+3}+\dots  \end{array}
\right]=
\]
\begin{equation} \label{9}
=\left[ \begin{array}{ccc}
(-a_0)t^{n}+(-a_1)t^{n+1}+(-a_2+b_2)t^{n+2}+(-a_3+b_3)t^{n+3}+\dots\\ 
(-a_0)t^{n}+(-a_1)t^{n+1}+(-a_2+c_2)t^{n+2}+(-a_3+c_3)t^{n+3}+\dots\\ 
(-a_0)t^{n}+(-a_1)t^{n+1}+(-a_2)t^{n+2}+(-a_3)t^{n+3}+\dots~~~~~~~~~~~~ 
\end{array}
\right].
\end{equation}
Therefore, $Tv \in X_2$;

{\bf 3.}
Let apply on  the vector $v$  given by \eqref{8} with $T^{-1}$: 
 \[
~~T^{-1}v=T^{-1}\left[ \begin{array}{ccc}
a_0t^{n}+a_1t^{n+1}+a_2t^{n+2}+a_3t^{n+3}+\dots  \\ 
~~~~~~~~~~~~~~~~~~~~~b_2t^{n+2}+b_3t^{n+3}+\dots  \\ 
~~~~~~~~~~~~~~~~~~~~~c_2t^{n+2}+c_3t^{n+3}+\dots  \end{array}
\right]=
\]
\begin{equation}  \label{10}
=\left[ \begin{array}{ccc}
~~~~~~~~~~~~~~~~~~~~(-c_2)t^{n+2}+(-c_3)t^{n+3}+\dots   \\ 
a_0t^{n}+a_1t^{n+1}+(a_2-c_2)t^{n+2}+(a_3-c_3)t^{n+3}+\dots\\ 
~~~~~~~~~~~~~~~~~~~~~~(b_2-c_2)t^{n+2}+(b_3-c_3)t^{n+3}+\dots  \end{array}
\right].
\end{equation}
Hence,  $T^{-1} v \in X_2$ for all $v \in X_1$; 

{\bf 4.} Let $v \in X_2$, then we have: 
\begin{equation} \label{11}
v=\left[ \begin{array}{ccc}
a_0t^{n}+a_1t^{n+1}+a_2t^{n+2}+\dots  \\ 
b_0t^{n}+b_1t^{n+1}+b_2t^{n+2}+\dots  \\ 
c_0t^{n}+c_1t^{n+1}+c_2t^{n+2}+\dots  \end{array}
\right],
\end{equation}
where $a_0,b_0,c_0\neq 0 \in \mathbb{Z}_p$. Therefore, 
\[
B^n v=B^n \left[ \begin{array}{ccc}
a_0t^{n}+a_1t^{n+1}+a_2t^{n+2}+\dots  \\ 
b_0t^{n}+b_1t^{n+1}+b_2t^{n+2}+\dots  \\ 
c_0t^{n}+c_1t^{n+1}+c_2t^{n+2}+\dots  \end{array}
\right]=
\]
\begin{equation} \label{12}
=B^{n-1}\left[ \begin{array}{ccc}
(-a_0)t^{n-1}+(-a_1+b_0)t^{n}+(-a_2+b_1)t^{n+1}+\dots~~~~~\\ 
~~~~~~~~~~~~b_0t^{n}+b_1t^{n+1}+\dots\\ 
~~~~~~~~~~~~~~~~~~~~b_0t^{n}+(b_1-c_0)t^{n+1}+\dots
\end{array}
\right].
\end{equation}
As we see, we have at least one more $B$ to multiply the obtained matrix by and therefore, the 
result belongs to $X_1$;

{\bf 5.} Let $v \in X_3$, then
\begin{equation} \label{13}
v=\left[ \begin{array}{ccc}
~~~~~~~~~~~~~~~~~~~~a_2t^{n+2}+a_3t^{n+3}+\dots   \\ 
b_0t^{n}+b_1t^{n+1}+b_2t^{n+2}+b_3t^{n+3}+\dots\\ 
~~~~~~~~~~~~~~~~~~~~~~c_2t^{n+2}+c_3t^{n+3}+\dots  \end{array}
\right].
\end{equation}
Consequently, we have
\[
B^nv=B^n \left[ \begin{array}{ccc}
~~~~~~~~~~~~~~~~~~~~a_2t^{n+2}+a_3t^{n+3}+\dots   \\ 
b_0t^{n}+b_1t^{n+1}+b_2t^{n+2}+b_3t^{n+3}+\dots\\ 
~~~~~~~~~~~~~~~~~~~~~~c_2t^{n+2}+c_3t^{n+3}+\dots  \end{array}
\right]=
\]
\begin{equation} \label{14}
=B^{n-1}\left[ \begin{array}{ccc}
~b_0t^{n}+(-a_2+b_1)t^{n+1}+(-a_3+b_2)t^{n+2}+\dots\\
b_0t^{n}+b_1t^{n+1}+b_2t^{n+2}+b_3t^{n+3}+\dots~~~~~~~~~\\
b_0t^{n}+b_1t^{n+1}+b_2t^{n+2}+(b_3-c_3)t^{n+3}+\dots  \end{array}
\right].
\end{equation}
Therefore, $Bv$ belongs to $X_2$ and by the previous case, we obtain that $B^nv \in X_1$ for all $n \ge 3$. 
\end{proof}

\begin{corollary} The matrices $A^3$ and $B^3$ generate a non--abelian free group of rank $2$ over $\mathbb{Z}_p[t,t^{-1}]$ for any integer $p > 1$.
\end{corollary} 

\subsection*{Acknowledgements}
The research was partially supported by the institutional scientific research project of Batumi Shota Rustaveli State University and the internship program of the ministry of education, culture, and sport of the autonomous republic of Adjara. The paper is written during A. Beridze's visiting at University of California, Santa Barbara (UCSB) from 15.09.2019 to 15.12.2019.

\end{document}